\documentclass[12pt, reqno]{amsart}
\usepackage{amsmath, amsthm, amscd, amsfonts, amssymb, graphicx, color, mathrsfs}
\usepackage[bookmarksnumbered, colorlinks, plainpages]{hyperref}
\usepackage[all]{xy}
\usepackage{slashed}

\usepackage{soul}
\usepackage{cancel}
\usepackage{ulem}

\newtheorem{theorem}{Theorem}[section]

\theoremstyle{definition}

\theoremstyle{remark}

\numberwithin{equation}{section}

\begin{document}
\setcounter{page}{1}

\title[  Nuclear pseudo-multipliers associated to the Harmonic oscillator ]{ Characterization of nuclear  pseudo-multipliers associated to the harmonic oscillator}

\author[D. Cardona]{Duv\'an Cardona}
\address{
  Duv\'an Cardona:
  \endgraf
  Department of Mathematics  
  \endgraf
  Pontificia Universidad Javeriana.
  \endgraf
  Bogot\'a
  \endgraf
  Colombia
  \endgraf
  {\it E-mail address} {\rm duvanc306@gmail.com;
cardonaduvan@javeriana.edu.co}
  }

 \author[E. Samuel Barraza]{E. Samuel Barraza}
\address{
 E. Samuel Barraza:
  \endgraf
  Department of computer science and Artificial intelligence. 
  \endgraf
  Universidad de Sevilla.
  \endgraf
 Sevilla
  \endgraf
 Spain
  \endgraf
  {\it E-mail address} {\rm edglibre@gmail.com;
edglibre@gmail.com; edgbarver@alum.us.es}
  }



\begin{abstract}
In this paper we  study pseudo-multipliers associated to the harmonic oscillator (also called Hermite multipliers) belonging to the ideal of $r$-nuclear operators on Lebesgue spaces.
\end{abstract} \maketitle
\section{Introduction}
\subsection{Outline of the paper} In this paper, we are interested in the $r$-nuclearity of  pseudo-multipliers associated to the harmonic oscillator (also called Hermite pseudo-multipliers) on $L^p(\mathbb{R}^n)$-spaces. This paper is the continuation of the work \cite{BarrazaCardona} where the authors have given necessary conditions for the $r$-nuclearity of Hermite multipliers. Now,  we recall some notions about pseudo-multipliers. Let us consider the sequence of Hermite function on $\mathbb{R}^n,$
\begin{equation}
\phi_\nu=\Pi_{j=1}^n\phi_{\nu_j},\,\,\, \phi_{\nu_j}(x_j)=(2^{\nu_j}\nu_j!\sqrt{\pi})^{-\frac{1}{2}}H_{\nu_j}(x_j)e^{-\frac{1}{2}x_j^2}
\end{equation} 
where $x=(x_1,\cdots,x_n)\in\mathbb{R}^n$, $\nu=(\nu_1,\cdots,\nu_n)\in\mathbb{N}^n_0,$ and $H_{\nu_j}(x_j)$ denotes the Hermite polynomial of order $\nu_j$. It is well known that the Hermite functions provide a complete and orthonormal system in $L^2(\mathbb{R}^n).$ If we consider the operator $L=-\Delta+|x|^{2}$ acting on the Schwartz space $\mathscr{S}(\mathbb{R}^n),$ where $\Delta$ is the standard Laplace operator on $\mathbb{R}^n,$ then we have the relation 
$
L\phi_\nu=\lambda_\nu \phi_\nu,\,\,\nu\in\mathbb{N}_0^n.
$
The operator $L$ is symmetric and positive in $L^2(\mathbb{R}^n)$ and admits a self-adjoint extension $H$ whose domain is given by
\begin{equation}
\textnormal{Dom}(H)=\left\{\sum_{\nu\in\mathbb{N}_0^n}\langle f,\phi_\nu \rangle_{L^2}\phi_\nu:   \sum_{\nu\in\mathbb{N}_0^n}|\lambda_\nu\langle f,\phi_\nu \rangle_{L^2}|^2 <\infty \right\}. 
\end{equation} So, for $f\in \textnormal{Dom}(H), $ we have
\begin{equation}
(Hf)(x)=\sum_{\nu\in\mathbb{N}_{0}}\lambda_\nu\widehat{f}(\phi_\nu)\phi_\nu(x),\,\,\,\, \widehat{f}(\phi_\nu)=\langle f,\phi_\nu \rangle_{L^2}.
\end{equation}
The operator $H$  is precisely the quantum harmonic oscillator on $\mathbb{R}^n$ (see \cite{Prugovecki}). The sequence $\{\widehat{f}(\phi_v)\} $ determines  the Fourier-Hermite transform of $f,$ with corresponding inversion formula
\begin{equation}\label{inversion}
f(x)=\sum_{\nu\in \mathbb{N}^n_0}\widehat{f}(\phi_v)\phi_\nu(x).
\end{equation}
On the other hand, pseudo-multipliers are defined by the quantization process that associates to a function $m$  on $\mathbb{R}^n\times\mathbb{N}_0^n$ a linear operator $T_m$ of the form:
\begin{equation}\label{pseudo}
T_mf(x)=\sum_{\nu\in\mathbb{N}^n_0}m(x,\nu)\widehat{f}(\phi_\nu)\phi_\nu(x),\,\,\,\,f\in \textnormal{Dom}(T_m).
\end{equation}
The  function $m$ on $\mathbb{R}^n\times \mathbb{N}_0^n$ is called the symbol of the pseudo-multiplier $T_m.$ If in \eqref{pseudo}, $m(x,\nu)=m(\nu)$ for all $x,$ the operator $T_m$ is called a multiplier. Multipliers and pseudo-multipliers have been studied, for example, in the works \cite{BagchiThangavelu,stempak,stempak1,stempak2,Thangavelu,Thangavelu2} (and references therein) principally by its mapping properties on $L^p$ spaces. In order that the operator $T_m:L^{p_1}(\mathbb{R}^n)\rightarrow L^{p_2}(\mathbb{R}^n)$ extends to a $r$-nuclear operator, in this paper we  provide necessary and sufficient conditions on the symbol $m.$

\subsection{Nuclearity of pseudo-multipliers} We recall the notion of $r$-nuclearity as follows.  By following A. Grothendieck \cite{GRO}, we can recall that a linear operator $T:E\rightarrow F$  ($E$ and $F$ Banach spaces) is  $r$-nuclear, if
there exist  sequences $(e_n ')_{n\in\mathbb{N}_0}$ in $ E'$ (the dual space of $E$) and $(y_n)_{n\in\mathbb{N}_0}$ in $F$ such that
\begin{equation}\label{nuc}
Tf=\sum_{n\in\mathbb{N}_0} e_n'(f)y_n,\,\,\, \textnormal{ and }\,\,\,\sum_{n\in\mathbb{N}_0} \Vert e_n' \Vert^r_{E'}\Vert y_n \Vert^r_{F}<\infty.
\end{equation}
\noindent The class of $r-$nuclear operators is usually endowed with the quasi-norm
\begin{equation}
n_r(T):=\inf\left\{ \left\{\sum_n \Vert e_n' \Vert^r_{E'}\Vert y_n \Vert^r_{F}\right\}^{\frac{1}{r}}: T=\sum_n e_n'\otimes y_n \right\}
\end{equation}
\noindent and, if $r=1$, $n_1(\cdot)$ is a norm and we obtain the ideal of nuclear operators. In addition, when $E=F$ is a Hilbert space and $r=1$ the definition above agrees with the concept of  trace class operators. For the case of Hilbert spaces $H$, the set of $r$-nuclear operators agrees with the Schatten-von Neumann class of order $r$ (see Pietsch  \cite{P,P2}).\\
\\
In order to study the $r$-nuclearity and the spectral trace of Hermite pseudo-multipliers, we will use results from J. Delgado \cite{D2}, on the characterization of nuclear integral  operators on $L^p(X,\mu)$ spaces, which in this case can be applied to $L^p$ spaces on $\mathbb{R}^n$. Indeed, we will prove that under certain conditions, a $r$-nuclear operator $T_m: L^{p}(\mathbb{R}^n)\rightarrow  L^{p}(\mathbb{R}^n)$ has a nuclear trace given by
\begin{eqnarray}
\textnormal{Tr}(T_m)=\int_{\mathbb{R}^n}\sum_{\nu\in\mathbb{N}_0^n}m(x,\nu)\phi_\nu(x)^2dx.
\end{eqnarray}

It was proved in  \cite{BarrazaCardona} that a multiplier $T_m$ with symbol satisfying one of the following conditions
\begin{itemize}
\item $1\leq p_2<4,$ $\frac{4}{3}<p_1<\infty$ and
\begin{equation}
\varkappa(m,p_1,p_2):=\sum_{s=0}^n\sum_{\nu\in I_s}k^{\frac{sr}{2}(\frac{1}{p_2}-\frac{1}{p_1})} (\prod_{\nu_j>k}\nu_j )^{\frac{r}{2}(\frac{1}{p_2}-\frac{1}{p_1})}|m(\nu)|^r<\infty,
\end{equation}
\item $1\leq p_2<4,$  $p_1=\frac{4}{3}$ and
\begin{equation}
\varkappa(m,p_1,p_2):=\sum_{s=0}^n\sum_{\nu\in I_s}{k}^{\frac{sr}{2}(\frac{1}{p_2}-\frac{3}{4})}(\ln k)^{sr}\cdot  \prod_{\nu_j>k}[{\nu_j}^{\frac{r}{2}(\frac{1}{p_2}-\frac{3}{4})}(\ln(\nu_j))^r]|m(\nu)|^r<\infty,
\end{equation}
\item  $1\leq p_2<4,$  $1<p_1<\frac{4}{3}$ and
\begin{equation}
\varkappa(m,p_1,p_2):=\sum_{s=0}^n\sum_{\nu\in I_s} k^{\frac{sr}{2}(\frac{1}{p_2}+\frac{1}{3p_1}-1)}\cdot (\prod_{\nu_j>k} \nu_j)^{\frac{r}{2}(\frac{1}{p_2}+\frac{1}{3p_1}-1)}|m(\nu)|^r<\infty,
\end{equation}
\item $p_2=4,$ $\frac{4}{3}<p_1<\infty$ and
\begin{equation}
\varkappa(m,p_1,p_2):=\sum_{s=0}^{n} \sum_{\nu\in {I}_s}k^{\frac{sr}{2}(\frac{1}{4}-\frac{1}{p_1})}(\ln (k))^{sr}\prod_{\nu_j>k}[(\ln(\nu_j))^r\nu_j^{\frac{r}{2}(\frac{1}{4}-\frac{1}{p_1})}]|m(\nu)|^r<\infty,
\end{equation}
\item $p_2=4,$ $p_1=\frac{4}{3}$ and
\begin{equation} \varkappa(m,p_1,p_2):=\sum_{s=0}^{n}
\sum_{\nu\in {I}_s}k^{-\frac{sr}{4}}(\ln k)^{2sr}\prod_{\nu_j>k}[\nu_j^{-\frac{r}{4}}(\ln \nu_j)^{2r}]\cdot |m(\nu)|^r<\infty,
\end{equation}
\item $p_2=4,$ $1<p_1<\frac{4}{3}$ and
\begin{equation}
\varkappa(m,p_1,p_2):=\sum_{s=0}^{n} \sum_{\nu\in I_s}k^{\frac{sr}{6}(\frac{1}{p_1}-\frac{9}{4})}(\ln(k))^{sr}\prod_{\nu_j>k}[\nu_j^{\frac{r}{6}(\frac{1}{p_1}-\frac{9}{4})}\ln(\nu_j)^r  ]\cdot |m(\nu)|^r<\infty,
\end{equation}
\item  $4<p_2\leq \infty,$ $\frac{4}{3}<p_1<\infty $ and
\begin{equation}
\varkappa(m,p_1,p_2):=\sum_{s=0}^{n} \sum_{\nu\in I_s}k^{\frac{sr}{2}(\frac{1}{3p_2'}-\frac{1}{p_1})}(\prod_{\nu_j>k}\nu_j)^{\frac{r}{2}(\frac{1}{3p_2'}-\frac{1}{p_1})}|m(\nu)|^r<\infty,
\end{equation}
\item  $4<p_2\leq \infty,$ $p_1=\frac{4}{3}$ and
\begin{equation}
\varkappa(m,p_1,p_2):=\sum_{s=0}^{n} \sum_{\nu\in I_s}k^{-\frac{sr}{6}(\frac{1}{p_2}+\frac{5}{4})}(\ln(k))^{sr}\prod_{\nu_j>k}[  \nu_j^{-\frac{r}{6}(\frac{1}{p_2}+\frac{5}{4})}(\ln(\nu_j))^{r}  ]|m(\nu)|^r<\infty,
\end{equation}
\item $4<p_2\leq \infty,$  $1<p_1<\frac{4}{3}$ and
\begin{equation}
\varkappa(m,p_1,p_2):=\sum_{s=0}^{n} \sum_{\nu\in I_s}k^{\frac{sr}{6}(\frac{1}{p_1}-\frac{1}{p_2}-2)}\cdot (\prod_{\nu_j>k}\nu_j)^{\frac{r}{6}(\frac{1}{p_1}-\frac{1}{p_2}-2)}|m(\nu)|^r<\infty,
\end{equation}
\end{itemize}
where $\{I_s\}_{s=0}^n$ is a suitable partition of $\mathbb{N}_0^n,$ can be extended to a $r$-nuclear operator from $L^{p_1}(\mathbb{R}^n)$ into $L^{p_2}(\mathbb{R}^n).$ Although is easy to see that similar necessary conditions apply for pseudo-multipliers, and that such conditions can be useful for applications because they can verified, for example, numerically for $m$ given, in this paper we want to characterize the $r$-nuclearity of pseudo-multipliers by using abstract conditions depending on the existence of certain measurable functions. In fact, the main result of this paper is the following.

\begin{theorem}
 $T_{m}:L^{p_1}(\mathbb{R}^n)\rightarrow L^{p_2}(\mathbb{R}^n)$ extends to a $r$-nuclear operator, if and only if, for every $\nu\in\mathbb{N}_0^n,$ the function $m(\cdot,\nu)\phi_\nu$ admits a decomposition of the form
\begin{equation}\label{admit}
m(x,\nu)\phi_{\nu}(x)=\sum_{k=1}^{\infty}h_{k}(x)\widehat{g}(\phi_\nu)
\end{equation} where   $\{g_k\}_{k\in\mathbb{N}}$ and $\{h_k\}_{k\in\mathbb{N}}$ are sequences of functions satisfying 
\begin{equation}
\sum_{k=0}^{\infty}\Vert g_k\Vert^r_{L^{p_1'}}\Vert h_{k}\Vert^r_{L^{p_2}}<\infty.
\end{equation}
\end{theorem}
Some remarks about our main theorem is the following.
\begin{itemize}
\item A consequence of the above theorem is that symbols associated to nuclear multipliers admit decomposition of the form $
m(\nu)=\sum_{k=0}^{\infty}\widehat{h}_{k}(\phi_\nu)\widehat{g}_{k}(\phi_\nu).
$ This can be obtained multiplying both sides of \eqref{admit} by $\phi_\nu$ and later integrating both sides over $\mathbb{R}^n.$
\item Our approach is an adaptation to the non-compact case of $\mathbb{R}^n$ of techniques used in the  work \cite{Ghaemi} by M. B. Ghaemi, M. Jamalpour Birgani, and M. W. Wong.
\item For every $\nu,$ the function $\phi_\nu$ has only  finitely many zeros. So, the set $M=\{x: \phi_\nu(x)=0\textnormal{ for some }\nu\}$ is a countable subset of $\mathbb{R}^n.$ According to \eqref{admit}, outside of the set $M$ we have 
\begin{equation}\label{admit'}
m(x,\nu)=\phi_{\nu}(x)^{-1}\sum_{k=1}^{\infty}h_{k}(x)\widehat{g}(\phi_\nu).
\end{equation} 
\end{itemize}

\subsection{Related works}
Now, we include some references on the subject. Sufficient conditions for the  $r$-nuclearity of spectral multipliers associated to the harmonic oscillator, but, in modulation spaces and Wiener amalgam spaces have been considered by  J. Delgado, M. Ruzhansky and B. Wang in \cite{DRB,DRB2}. The Properties of these multipliers in $L^p$-spaces have been investigated in the references S. Bagchi, S. Thangavelu  \cite{BagchiThangavelu}, J. Epperson \cite{Epperson},  K. Stempak and J.L. Torrea \cite{stempak,stempak1,stempak2},  S. Thangavelu \cite{Thangavelu,Thangavelu2} and references therein. Hermite expansions for distributions can be found in B. Simon \cite{Simon}. The $r$-nuclearity and Grothendieck-Lidskii formulae for multipliers and other types of integral operators can be found in \cite{D3,DRB2}. Sufficient conditions for the nuclearity of  pseudo-differential operators on the torus can be found in \cite{DW,Ghaemi}. The references \cite{DR,DR1,DR3,DR5} and \cite{DRTk} include a complete study on the $r$-nuclearity, $0<r\leq 1,$ of multipliers (and pseudo-differential operators) on compact Lie groups, and more generally on compact manifolds, with explicit conditions on  symbols of operators providing an useful tool for applications (see \cite{CardonaDelCorral}). For compact and Hausdorff groups, the  work \cite{Ghaemi2} by M. B. Ghaemi, M. Jamalpour Birgani, and M. W. Wong characterize in terms of the existence of certain measurable functions the nuclearity of pseudo-differential operators. On Hilbert spaces the class of $r$-nuclear operators agrees with the Schatten-von Neumann class $S_{r}(H);$ in this context operators with integral kernel on Lebesgue spaces and, in particular, operators with  kernel acting of a special way with anharmonic oscillators of the form $E_a=-\Delta_x+|x|^a,$ $a>0,$ has been considered on  Schatten classes on $L^2(\mathbb{R}^n)$ in J. Delgado and M. Ruzhansky \cite{kernelcondition}.

The proof of our  results will be presented in the next section.

\section{Nuclear pseudo-multipliers associated to the harmonic oscillator}
\subsection{Characterization of nuclear pseudo-multipliers}
In this section we prove our main result for pseudo-multipliers $T_m$. Our criteria will be formulated in terms of the symbols $m.$ First, let us observe that every multiplier $T_m$ is an operator with kernel $K_m(x,y).$ In fact, straightforward computation show that
\begin{equation}\label{kernelpseudo}
T_mf(x) =\int_{\mathbb{R}^n}K_m(x,y)f(y)dy,\,\,K_m(x,y):=\sum_{\nu\in\mathbb{N}^n_0}m(x,\nu)\phi_\nu(x)\phi_{\nu}(y)
\end{equation}
for every  $f\in\mathscr{D}(\mathbb{R}^n).$ In order to analyze the $r$-nuclearity of $T_m$ we study its kernel $K_m$ by using the following theorem (see J. Delgado \cite{Delgado,D2}).
\begin{theorem}\label{Theorem1} Let us consider $1\leq p_1,p_2<\infty,$ $0<r\leq 1$ and let $p_1'$ be such that $\frac{1}{p_1}+\frac{1}{p_1'}=1.$ An operator $T:L^p(\mu_1)\rightarrow L^p(\mu_2)$ is $r$-nuclear if and only if there exist sequences $(g_n)_n$ in $L^{p_2}(\mu_2),$ and $(h_n)$ in $L^{q_1}(\mu_1),$ such that
\begin{equation}
\sum_{n}\Vert g_n\Vert_{L^{p_2}}^r\Vert h_n\Vert^r_{L^{q_1} } <\infty,\textnormal{        and        }Tf(x)=\int(\sum_ng_{n}(x)h_n(y))f(y)d\mu_1(y),\textnormal{   a.e.w. }x,
\end{equation}
for every $f\in {L^{p_1}}(\mu_1).$ In this case, if $p_1=p_2$ $($\textnormal{see Section 3 of} \cite{Delgado}$)$ the nuclear trace of $T$ is given by
\begin{equation}\label{trace1}
\textnormal{Tr}(T):=\int\sum_{n}g_{n}(x)h_{n}(x)d\mu_1(x).
\end{equation}
\end{theorem}
Now, we prove our main theorem.
\begin{theorem}
 Let $0<r\leq 1.$ The operator $T_{m}:L^{p_1}(\mathbb{R}^n)\rightarrow L^{p_2}(\mathbb{R}^n)$ extends to a $r$-nuclear operator, if and only if, for every $\nu\in\mathbb{N}_0^n,$ the function $m(\cdot,\nu)\phi_\nu$ admits a decomposition of the form
\begin{equation}
m(x,\nu)\phi_{\nu}(x)=\sum_{k=1}^{\infty}h_{k}(x)\widehat{g}(\phi_\nu)
\end{equation} where   $\{g_k\}_{k\in\mathbb{N}}$ and $\{h_k\}_{k\in\mathbb{N}}$ are sequences of functions satisfying 
\begin{equation}
\sum_{k=0}^{\infty}\Vert g_k\Vert^r_{L^{p_1'}}\Vert h_{k}\Vert^r_{L^{p_2}}<\infty.
\end{equation}
\end{theorem}
\begin{proof}
Let us assume that $T_{m}:L^{p_1}(\mathbb{R}^n)\rightarrow L^{p_2}(\mathbb{R}^n)$ is a $r$-nuclear operator. Then there exist sequences $h_{k}$ in $L^{p_2}$ and $g_k$ in $L^{p_1'}$ satisfying
\begin{equation}
Tf(x)=\int_{\mathbb{R}^n}\left(\sum_{k=1}^{\infty}h_k(x)g_{k}(y)\right)f(y)dy,\,\, f\in L^{p_1},
\end{equation} with 
\begin{equation}\label{deco2}
\sum_{k=0}^{\infty}\Vert g_k\Vert^r_{L^{p_1'}}\Vert h_{k}\Vert^r_{L^{p_2}}<\infty.
\end{equation}
Since every Hermite function $\phi_\nu$ belongs to the Schwartz class which is contain in $L^p-$spaces,  for $f=\phi_\nu\in L^{p_1}$  we have
\begin{align*}
T_{m}(\phi_\nu) =\int_{\mathbb{R}^n}\left(\sum_{k=1}^{\infty}h_k(x)g_{k}(y)\right)\phi_\nu(y)dy
=\sum_{k=1}^{\infty}h_k(x)\widehat{g}_{k}(\nu).
\end{align*} Now, if we compute $T_m(\phi_\nu)$ from the definition of pseudo-multipliers \eqref{pseudo}, we obtain
\begin{equation}
T_{m}(\phi_\nu)(x)=m(x,\nu)\phi_{\nu}(x)
\end{equation} where we have used the $L^2$-orthogonality of Hermite functions. Consequently, we deduce the identity
\begin{equation}\label{deco}
m(x,\nu)\phi_{\nu}(x)=\sum_{k=1}^{\infty}h_k(x)\widehat{g}_{k}(\nu).
\end{equation} So, we have proved the first part of the theorem. Now, if we assume that the symbol $m$ of  a multiplier $T_m$ and every Hermite function $\phi_\nu$ satisfies the decomposition formula \ref{deco} for fixed sequences  $h_{k}$ in $L^{p_2}$ and $g_k$ in $L^{p_1'}$ satisfying \eqref{deco2}, then from \eqref{pseudo} we can write
\begin{align*}
T_{m}f(x) &=\sum_{\nu\in\mathbb{N}_0^n}m(x,\nu)\phi_{\nu}(x)\widehat{f}(\phi_\nu)=\sum_{\nu\in\mathbb{N}_0^n}\sum_{k=1}^{\infty}h_k(x)\widehat{g}_{k}(\nu)\widehat{f}(\phi_\nu)\\
&=\sum_{\nu\in\mathbb{N}_0^n}\sum_{k=1}^{\infty}h_k(x)\int_{\mathbb{R}^n}g_k(y)\phi_{\nu}(y)dy\widehat{f}(\phi_\nu)\\
&= \int_{\mathbb{R}^n}\left(\sum_{k=1}^{\infty}h_k(x)g_k(y)\right)\left(\sum_{\nu\in\mathbb{N}_0^n}\widehat{f}(\phi_\nu)\phi_{\nu}(y)\right)dy\\
&= \int_{\mathbb{R}^n}\left(\sum_{k=1}^{\infty}h_k(x)g_k(y)\right)f(y)dy,
\end{align*}
where in the last line we have used the inversion formula \eqref{inversion}. So, by  Delgado Theorem (Theorem \ref{Theorem1}) we end the proof.
\end{proof}

\subsection{Traces of nuclear pseudo-multipliers of the harmonic oscillator} If  $T:E\rightarrow E$ is  $r$-nuclear, with the Banach space $E$ satisfying the Grothendieck approximation property (see Grothendieck\cite{GRO}),  then 
there exist  sequences $(e_n ')_{n\in\mathbb{N}_0}$ in $ E'$ (the dual space of $E$) and $(y_n)_{n\in\mathbb{N}_0}$ in $E$ such that
\begin{equation}\label{nuc2}
Tf=\sum_{n\in\mathbb{N}_0} e_n'(f)y_n,\,\,\,\,\,\textnormal{and}\,\,\,\,\,
\sum_{n\in\mathbb{N}_0} \Vert e_n' \Vert^r_{E'}\Vert y_n \Vert^r_{F}<\infty.
\end{equation}
In this case the nuclear trace of $T$ is given by
$
\textnormal{Tr}(T)=\sum_{n\in\mathbb{N}^n_0}e_n'(f_n).
$
$L^p$-spaces have the Grothendieck approximation property and as consequence we can compute the nuclear trace of every $r$-nuclear pseudo-multipliers. For to do so, let us consider a $r-$nuclear pseudo-multiplier  $T_m:L^{p}(\mathbb{R}^n)\rightarrow L^{p}(\mathbb{R}^n).$ Since the function, $\varkappa(x,y):=\sum_{k}h_k(x)g_k(y)$ is defined a.e.w., let us choose $z\in \mathbb{R}^n$ such that $\varkappa(x,z)$ is finite $a.e.w.$ Let us consider  $B(z,r),$ the ball centered at $z$ with radius $r>0.$ Let us denote by $|B(z,r)|$ the Lebesgue measure of $B(z,r).$ If $f=|B(z,r)|^{-1}\cdot 1_{B(z,r)},$ where $ 1_{B(z,r)}$ is the characteristic function of $B(z,r),$ we obtain
\begin{equation}
T_m(|B(z,r)|^{-1}\cdot 1_{B(z,r)}) =\frac{1}{|B(z,r)|}\int_{B(z,r)}\left(\sum_{k=1}^{\infty}h_k(x)g_k(y)\right)dy
\end{equation}
but, we also have
\begin{eqnarray}
T_{m}( |B(z,r)|^{-1}\cdot 1_{B(z,r)} )=\frac{1}{|B(z,r)|}\int_{B(z,r)}K_m(x,y)dy,
\end{eqnarray}
where $K_m$ is defined as in \eqref{kernelpseudo}. So, we have the identity
\begin{eqnarray}
\frac{1}{|B(z,r)|}\int_{B(z,r)}K_m(x,y)dy=\frac{1}{|B(z,r)|}\int_{B(z,r)}\left(\sum_{k=1}^{\infty}h_k(x)g_k(y)\right)dy
\end{eqnarray}
for every $r>0.$ Taking limit as $r\rightarrow 0^+$ and by applying Lebesgue differentiation Theorem, we obtain
\begin{eqnarray}
K_{m}(x,z)=\sum_{k=1}^{\infty}h_k(x)g_k(z),\,\,a.e.w.
\end{eqnarray} Finally, the nuclear trace of $T_m$ can be computed from \eqref{trace1}. So, we have
\begin{equation}
\textnormal{Tr}(T_m)=\int_{\mathbb{R}^n}K_{m}(x,x)dx=\int_{\mathbb{R}^n}\sum_{\nu\in\mathbb{N}_0^n}m(x,\nu)\phi_\nu(x)^2dx.
\end{equation}

Now, in order to determinate a relation with the eigenvalues of $T_m$ we recall that, the nuclear trace of an $r$-nuclear operator  on a Banach space coincides with the spectral trace, provided that $0<r\leq \frac{2}{3}.$ For $\frac{2}{3}\leq r\leq 1$ we recall the following result (see \cite{O}).
\begin{theorem} Let $T:L^p(\mu)\rightarrow L^p(\mu)$ be a $r$-nuclear operator as in \eqref{nuc2}. If $\frac{1}{r}=1+|\frac{1}{p}-\frac{1}{2}|,$ then, 
\begin{equation}
\textnormal{Tr}(T):=\sum_{n\in\mathbb{N}^n_0}e_n'(f_n)=\sum_{n}\lambda_n(T)
\end{equation}
where $\lambda_n(T),$ $n\in\mathbb{N}$ is the sequence of eigenvalues of $T$ with multiplicities taken into account. 
\end{theorem}
As an immediate consequence of the preceding theorem, if $T_m:L^p(\mathbb{R}^n)\rightarrow L^p(\mathbb{R}^n)$ is a $r$-nuclear  pseudo-multiplier and $\frac{1}{r}=1+|\frac{1}{p}-\frac{1}{2}|$ then, 
\begin{equation}
\textnormal{Tr}(T_m)=\int_{\mathbb{R}^n}\sum_{\nu\in\mathbb{N}_0^n}m(x,\nu)\phi_\nu(x)^2dx.=\sum_{n}\lambda_n(T),
\end{equation}
where $\lambda_n(T),$ $n\in\mathbb{N}$ is the sequence of eigenvalues of $T_m$ with multiplicities taken into account.

\bibliographystyle{amsplain}

\begin{thebibliography}{99}

\bibitem{BagchiThangavelu}  Bagchi, S. Thangavelu, S. On Hermite pseudo-multipliers. J. Funct. Anal. 268 (1), 140–-170, (2015)

\bibitem{BarrazaCardona} Barraza, E.S., Cardona, D. On nuclear $L^p$-multipliers associated to the Harmonic oscillator, in: Analysis in Developing Countries, Springer Proceedings in Mathematics $\&$ Statistics, Springer, 2018, M. Ruzhansky and J. Delgado (Eds), to appear.  

\bibitem{Cardona} Cardona D. Nuclear pseudo-differential operators in Besov spaces on compact Lie groups; to appear in J. Fourier Anal. Appl. 2017.

\bibitem{CardonaDelCorral} Cardona, D., Del Corral, C. The Dixmier trace and the non-commutative residue for multipliers on compact manifolds,  submitted. arXiv:1703.07453 

\bibitem{Delgado} Delgado, J.: A trace formula for nuclear operators on $L^p$, in: Schulze, B.W., Wong, M.W. (eds.)
Pseudo-Differential Operators: Complex Analysis and Partial Differential Equations, Operator Theory:
Advances and Applications, 205,  181-–193. Birkhäuser, Basel (2010)

\bibitem{DW} Delgado, J., Wong, M.W.: $L^p$-nuclear pseudo-differential operators on $\mathbb{Z}$ and $\mathbb{S}^1.,$  Proc. Amer. Math. Soc., 141 (11), 3935--394, (2013) 

\bibitem{D2} J. Delgado. The trace of nuclear operators on $L^p(\mu)$ for $\sigma$-finite Borel measures on second countable spaces. Integral Equations Operator Theory, 68(1), 61–-74, (2010)

\bibitem{D3} Delgado, J.: On the $r$-nuclearity of some integral operators on Lebesgue spaces. Tohoku Math. J. 67(2),  no. 1, 125--135, (2015)


\bibitem{DRB} Delgado, J. Ruzhansky, M. Wang, B. Approximation property and nuclearity on mixed-norm $L^p$, modulation and Wiener amalgam spaces. J. Lond. Math. Soc.  94, 391--408, (2016)

\bibitem{DRB2} Delgado, J. Ruzhansky, M. Wang, B. Grothendieck-Lidskii trace formula for mixed-norm $L^p$ and variable Lebesgue spaces. to appear in  J. Spectr. Theory. 

\bibitem{DR} Delgado, J. Ruzhansky, M.: { $L^p$-nuclearity, traces, and Grothendieck-Lidskii formula on compact Lie groups.,} J. Math. Pures Appl. \textbf{(9)}, 102(1),  153-172 (2014)

\bibitem{DR1}  Delgado, J. Ruzhansky, M.: Schatten classes on compact manifolds: Kernel conditions. J.  Funct. Anal., 267(3), 772--798, (2014)

\bibitem{DR3} Delgado, J. Ruzhansky, M.: Kernel and symbol criteria for Schatten classes and r-nuclearity on compact manifolds., C. R. Acad. Sci. Paris. Ser. I. 352.  779--784 (2014)


\bibitem{DR5} Delgado, J. Ruzhansky, M. Fourier multipliers, symbols and nuclearity on compact manifolds.  arXiv:1404.6479

\bibitem{kernelcondition}  Delgado, J. Ruzhansky, M. Schatten-von Neumann classes of integral operators. 	arXiv:1709.06446.

\bibitem{DRTk}  Delgado, J. Ruzhansky, M. Tokmagambetov, N. Schatten classes, nuclearity and nonharmonic analysis on compact manifolds with boundary.  arXiv:1505.02261

\bibitem{Epperson}  Epperson, J. Hermite multipliers and pseudo-multipliers, Proc. Amer. Math. Soc. 124(7)  2061-–2068, (1996)

\bibitem{Ghaemi} Ghaemi, M. B., Jamalpour Birgani, M., Wong, M. W. Characterizations of nuclear pseudo-differential operators on $\mathbb{S}^1$ with applications to adjoints and products. J. Pseudo-Differ. Oper. Appl. 8(2), 191–-201, (2017) 

\bibitem{Ghaemi2} Ghaemi, M. B., Jamalpour Birgani, M., Wong, M. W. Characterization, adjoints and products of nuclear pseudo-differential operators on compact and Hausdorff groups. U.P.B. Sci. Bull., Series A, Vol. 79(4), 207-220.  (2017) 

\bibitem{GRO0}  Grothendieck, A. La theorie de Fredholm, Bull. Soc. Math. France 84, 319–-384, (1956)

\bibitem{GRO} Grothendieck, A.: Produits tensoriels topologiques et espaces nucl\'eaires, Memoirs
Amer. Math. Soc. 16, Providence, 1955 (Thesis, Nancy, 1953).

\bibitem{Koch} Koch, H., Tataru, D. $L^p$-eigenfunction bounds for the Hermite operator. Duke
Math. J. 128, 369–-392. (2005)

\bibitem{Prugovecki} Prugove\u{c}ki, E. Quantum mechanics in Hilbert space. Second edition. Pure and Applied Mathematics, 92. Academic Press, Inc, New York-London, 1981.

\bibitem{P} Pietsch, A. Operator ideals. Mathematische Monographien, 16. VEB Deutscher Verlag der Wissenschaften, Berlin, 1978. 

\bibitem{P2} Pietsch, A. History of Banach spaces and linear operators. Birkh\"auser Boston, Inc., Boston, MA, 2007.

\bibitem{O} Reinov, O.I.,  Latif, Q.,  Grothendieck-Lidskii theorem for subspaces of Lp−spaces. Math. Nachr., Volume 286, Issue 2-3, 279--282, (2013).

\bibitem{NicolaRodino2010}Nicola, F., Rodino, L. Global pseudo-differential calculus on Euclidean spaces. Pseudo-Differential Operators. Theory and Applications, 4. Birkhäuser Verlag, Basel, 2010. 

\bibitem{Simon} Simon, B. Distributions and their Hermite expansions. J. Math. Phys. 12, 140--148 (1971)

\bibitem{stempak} Stempak, K. Multipliers for eigenfunction expansions of some Schrödinger operators, Proc. Amer. Math. Soc. 93, 477--482 (1985)

\bibitem{stempak1} Stempak, K., Torreą, J.L. On g-functions for Hermite function expansions, Acta Math. Hung. 109, 99--125,  (2005)

\bibitem{stempak2}  Stempak, K., Torreą, J.L. BMO results for operators associated to Hermite expansions, Illinois J. Math. 49, 1111--1132, (2005)

\bibitem{Thangavelu} Thangavelu, S. Lectures on Hermite and Laguerre Expansions, Math. Notes, vol. 42, Princeton University Press, Princeton, 1993.
\bibitem{Thangavelu2} Thangavelu, S.  Hermite and special Hermite expansions revisited Duke Mathematical Journal, 94(2), 257--278 (1998)



\end{thebibliography}

\end{document}